\documentclass[a4paper,twopage,reqno,12pt]{amsart}
\usepackage[top=30mm,right=30mm,bottom=30mm,left=30mm]{geometry}

\usepackage{graphicx}
\usepackage{amsmath}
\usepackage{amssymb}
\usepackage{amsfonts}
\usepackage{amsthm}
\usepackage{amstext}
\usepackage{amsbsy}
\usepackage{amsopn}
\usepackage{amscd}
\usepackage{enumerate}
\usepackage{color}
\usepackage[colorlinks]{hyperref}
\usepackage[hyperpageref]{backref}
\usepackage{multirow}
\usepackage{float}
\usepackage[norelsize]{algorithm2e}
\usepackage{url}

\newtheorem{theorem}{Theorem}[section]

\newtheorem{lemma}[theorem]{Lemma}

\theoremstyle{definition}

\numberwithin{equation}{section}

\newcommand{\GU}{\mathrm{GU}}

\newcommand{\SL}{\mathrm{SL}}
\newcommand{\SO}{\mathrm{SO}}
\newcommand{\SU}{\mathrm{SU}}
\newcommand{\SP}{\mathrm{Sp}}
\newcommand{\PSL}{\mathrm{PSL}}
\newcommand{\PSU}{\mathrm{PSU}}
\newcommand{\PSp}{\mathrm{PSp}}

\newcommand{\PGaO}{\mathrm{P\Gamma O}}

\newcommand{\POm}{\mathrm{P \Omega}}
\newcommand{\Sp}{\mathrm{Sp}}

\newcommand{\Aut}{\mathrm{Aut}}
\newcommand{\Out}{\mathrm{Out}}
\newcommand{\Inn}{\mathrm{Inn}}

\newcommand{\PG}{\mathrm{PG}}

\newcommand{\GAP}{\textsf{GAP}}

\newcommand{\F}{\mathbb{F}}

\newcommand{\Dmc}{\mathcal{D}}
\newcommand{\Pmc}{\mathcal{P}}
\newcommand{\Bmc}{\mathcal{B}}

\renewcommand{\leq}{\leqslant}
\renewcommand{\geq}{\geqslant}

\renewcommand{\mod}[1]{\ (\mathrm{mod}{\ #1})}
\newcommand{\imod}[1]{\allowbreak\mkern4mu({\operator@font mod}\,\,#1)}

\begin{document}
 \title{Symmetric designs and four dimensional projective special unitary groups}

 \author[S.H. Alavi]{Seyed Hassan Alavi}%
 \address{Seyed Hassan Alavi, Department of Mathematics, Faculty of Science, Bu-Ali Sina University, Hamedan, Iran.
 }%
 \email{alavi.s.hassan@basu.ac.ir and  alavi.s.hassan@gmail.com (G-mail is preferred)}
 \author[M. Bayat]{Mohsen Bayat}%
 \address{Mohsen Bayat, Department of Mathematics, Faculty of Science, Bu-Ali Sina University, Hamedan, Iran.}%
 \email{mohsen0sayeq24@gmail.com}
 \author[A. Daneshkhah]{Asharf Daneshkhah}%
 \thanks{Corresponding author: Asharf Daneshkhah}
 \address{Asharf Daneshkhah, Department of Mathematics, Faculty of Science, Bu-Ali Sina University, Hamedan, Iran.
 }%
  \email{adanesh@basu.ac.ir and daneshkhah.ashraf@gmail.com (G-mail is preferred)}
  \author[Sh. Zang Zarin]{Sheyda Zang Zarin}%
 \address{Sheyda Zang Zarin, Department of Mathematics, Faculty of Science, Bu-Ali Sina University, Hamedan, Iran.}%
 \email{zarinsheyda@gmail.com}

 \date{\today}%
 \dedicatory{Dedicated to Cheryl E. Praeger on the occasion of her 70th birthday}%

\begin{abstract}
    In this article, we study symmetric $(v, k, \lambda)$ designs admitting a flag-transitive and point-primitive automorphism group $G$ whose socle is $\PSU_{4}(q)$. We prove that there exist eight non-isomorphic such designs for which $\lambda\in\{3,6,18\}$ and $G$ is either $\PSU_{4}(2)$, or $\PSU_{4}(2):2$.
\end{abstract}

\maketitle
A \emph{symmetric $(v,k,\lambda)$ design} is an incidence structure $\Dmc=(\Pmc,\Bmc)$  consisting of a set $\Pmc$ of $v$ \emph{points} and a set $\Bmc$ of $v$ \emph{blocks} such that every point is incident with exactly $k$ blocks, and every pair of blocks is incident with exactly $\lambda$ points. A \emph{nontrivial} symmetric design is one in which $2<k<v-1$.
A \emph{flag} of $\Dmc$ is an incident pair $(\alpha,B)$, where $\alpha$ and $B$ are a point and a block of $\Dmc$, respectively. An \emph{automorphism} of a symmetric design $\Dmc$ is a permutation of the points permuting the blocks and preserving the incidence relation. An automorphism group $G$ of $\Dmc$ is called \emph{flag-transitive} if it is transitive on the set of flags of $\Dmc$. If $G$ is primitive on the point set $\Pmc$, then $G$ is said to be \emph{point-primitive}. The \emph{complement} of a symmetric $(v,k,\lambda)$ design $\Dmc$ is the symmetric $(v,v-k,v-2k+\lambda)$ design  whose set of points is the same as the set of points of $\Dmc$ and whose blocks are the complements of the blocks of $\Dmc$, that is, incidence is replaced by non-incidence and vice versa.
We here adopt the standard notation as in \cite{b:BHR-Max-Low,b:Atlas} for finite simple groups of Lie type, for example, we use $\PSL_{n}(q)$, $\PSp_{n}(q)$, $\PSU_{n}(q)$, $\POm_{2n+1}(q)$ and $\POm_{2n}^{\pm}(q)$ to denote the finite classical simple groups. A group $G$ is said to be \emph{almost simple} with socle $X$ if $X\unlhd G\leq \Aut(X)$, where $X$ is a nonabelian simple group. Symmetric and alternating groups on $n$ letters are denoted by $S_{n}$ and $A_{n}$, respectively.  We denote by $n$ the cyclic group of order $n$, and we write $E_{n}$ for an elementary abelian group of order $n$. Further notation and definitions in both design theory and group theory are standard and can be found, for example in \cite{b:beth1999design,b:Atlas,b:Dixon,b:Hugh-design}. We also use the software \textsf{GAP} \cite{GAP4} for computational arguments.

The main aim of this paper is to study flag-transitive and point-primitive symmetric designs.
It is known that if a nontrivial symmetric $(v, k, \lambda)$ design $\Dmc$ with $\lambda \leq 100$ admits a flag-transitive and point-primitive automorphism group $G$, then $G$ must be an affine or almost simple group~\cite{a:Zhou-lam100}. Therefore, it is somehow interesting to study such designs whose automorphism group $G$ is an almost simple group with socle $X$ being a finite simple group of low rank. In this direction, we recently studied such symmetric designs for $X$ a finite exceptional simple group \cite{a:ABD-EXP}. In the case where $X$ is a sporadic simple group, there exist four possible parameters (see \cite{a:Zhou-lam-large-sporadic}). For finite classical groups $X$, in \cite{a:ABD-PSL2}, we proved that there are only five possible symmetric $(v,k,\lambda)$ designs (up to isomorphism) admitting a flag-transitive and point-primitive almost simple automorphism group $G$ with socle $X=\PSL_{2}(q)$, see also \cite{a:Zhou-PSL2-q}.  This study for $X:=\PSL_{3}(q)$ gives rise to only one nontrivial design  which is a Desarguesian projective plane $\PG_{2}(q)$ and $\PSL_{3}(q) \leq G$ (see \cite{a:ABD-PSL3}), however when $X:=\PSU_{3}(q)$, there is no such non-trivial symmetric designs for $q\geq 4$, see \cite{a:D-PSU3}. This paper is devoted to studying symmetric designs admitting a flag-transitive and point-primitive almost simple automorphism group $G$ whose socle is $\PSU_{4}(q)$.

\begin{theorem}\label{thm:main}
Let $\Dmc=(\Pmc,\Bmc)$ be a symmetric $(v,k,\lambda)$ design with $\lambda \geq 1$, and let $\alpha\in \Pmc$. Suppose that $G$ is an automorphism group of $\Dmc$ whose socle is $X:=\PSU_{4}(q)$ with $q=p^{f}$. If $G$ is flag-transitive and point-primitive with $H:=G_{\alpha}$, then $X=\PSU_{4}(2)$ and $\lambda=3$, $6$, or $18$, and $v$, $k$, $\lambda$ and $G$ are as in one of the lines of {\rm Table~\ref{tbl:main}}.
\end{theorem}

\begin{table}
    \centering
    \scriptsize
    \caption{Symmetric designs admitting flag-transitive and point-primitive almost simple automorphism group with socle $\PSU_{4}(q)$.}\label{tbl:main}
    \begin{tabular}{p{3mm}p{2mm}p{2mm}p{2mm}p{10mm}llp{5cm}l}
        \hline\noalign{\smallskip}
        Line &
        \multicolumn{1}{c}{$v$} &
        \multicolumn{1}{c}{$k$} &
        \multicolumn{1}{c}{$\lambda$} &
        \multicolumn{1}{c}{$X$} &
        \multicolumn{1}{c}{$H$}  &
        \multicolumn{1}{c}{$G$}  &
        \multicolumn{1}{c}{Designs} &
        \multicolumn{1}{c}{References$^{\ast}$} \\
        \noalign{\smallskip}\hline\noalign{\smallskip}
        $1$ & $36$ & $15$ & $6$ & $\PSU_{4}(2)$  & $S_{6}$  & $\PSU_{4}(2)$
        & Menon
        &\cite{a:Braic-2500-nopower,a:rank3} \\
        $2$ & $36$ & $15$ & $6$ & $\PSU_{4}(2)$  & $S_{6}:2$  & $\PSU_{4}(2):2$
        & Menon
        &\cite{a:Braic-2500-nopower,a:rank3} \\
        $3$ & $40$ & $27$ & $18$ & $\PSU_{4}(2)$  & $3_{+}^{1+2}:2A_{4}$  & $\PSU_{4}(2)$
        &Complement of $\PG_{3}(3)$
        &\cite{a:Braic-2500-nopower,a:rank3} \\
        $4$ & $40$ & $27$ & $18$ & $\PSU_{4}(2)$  & $3_{+}^{1+2}:2A_{4}:2$  & $\PSU_{4}(2):2$
        &Complement of $\PG_{3}(3)$
        &\cite{a:Braic-2500-nopower,a:rank3} \\
        $5$ & $40$ & $27$ & $18$ & $\PSU_{4}(2)$  & $3^3:S_{4}$  & $\PSU_{4}(2)$
        & Complement of Higman design
        &\cite{a:Braic-2500-nopower,a:rank3} \\
        $6$ & $40$ & $27$ & $18$ & $\PSU_{4}(2)$  & $3^3:S_{4}:2$  & $\PSU_{4}(2):2$
        & Complement of Higman design
        &\cite{a:Braic-2500-nopower,a:rank3} \\
        $7$ & $45$ & $12$ & $3$ & $\PSU_{4}(2)$  & $2^{.}(A_{4}\times A_{4}).2$ & $\PSU_{4}(2)$
        &-
        &\cite{a:Braic-2500-nopower,a:rank3, a:Praeger-45-12-3}\\
        $8$ & $45$ & $12$ & $3$ & $\PSU_{4}(2)$  & $2^{.}(A_{4}\times A_{4}).2:2$ & $\PSU_{4}(2):2$
        &-
        &\cite{a:Braic-2500-nopower,a:rank3, a:Praeger-45-12-3}\\
        \noalign{\smallskip}\hline
        \multicolumn{9}{p{15cm}}{\tiny $\ast$ The last column addresses to references in which a design with the parameters in the line has been constructed.}
    \end{tabular}
\end{table}

\noindent \textbf{Comments on Table~\ref{tbl:main} }

\begin{description}
  \item[Lines 1-2] The symmetric $(36, 15, 6)$ designs are Menon design, that is  to say, symmetric designs with parameters $(4t^2,2t^2-t,t^2-t)$, for $t=3$. This designs can be constructed by orthogonal spaces. Let $(V, f )$ be a non-degenerate orthogonal space of dimension $2m+1$ over a finite field  $\F_{3}$ of size $3$ with discriminant $(-1)^m$, for $m>1$. The point set $\Pmc$ consists of all anisotropic $1$-dimensional subspaces $X =  \langle x \rangle \leq V$ satisfying $f (x,x)=1$, and blocks in $\Bmc$ have the form $B(X)=\{ Y\in \Pmc\mid  f (X,Y)=0\}$, for $X \in \Pmc$. Then $\Dmc=(\Pmc, \Bmc)$ is a symmetric design with parameters $(3^m(3^m -1)/2,3^{m-1}(3^m +1)/2,3^{m-1}(3^{m-1}+1)/2)$, $m>1$ (see~\cite{ a:Braic-2500-nopower,a:rank3}). This design is rank 3 with the full automorphism group $\POm_{2m+1}(3)$. If $m=2$, then we obtain the symmetric $(36, 15, 6)$ design with flag-transitive rank $3$ point-primitive automorphism group $\POm_{5}(3)\cong \PSU_{4}(2)$~\cite{a:Braic-2500-nopower,b:Atlas,a:rank3}.
  \item[Lines 3-4] These symmetric designs are the complement of the projective geometry  $\PG_{3}(3)$. The group $\PSp_{m+1}(q)$ with $m$ odd acts on $\PG_{m}(q)$ as a rank $3$ primitive group~\cite{a:rank3}.  For $m=3$ and $q=3$, we have the symmetric design $\PG_{m}(q)$ with parameters $(40, 13, 4)$ and rank $3$ point-primitive automorphism group $\PSp_{4}(3)\cong \PSU_{4}(2)$~\cite{a:Braic-2500-nopower,b:Atlas,a:rank3}. The complement of this design with parameters $(40, 27, 18)$ is flag-transitive.
  \item[Lines 5-6] These designs are orthogonal symmetric designs introduced by Higman \cite{a:Higman-rank3}, a series of designs with parameters $(\frac{q^{m+1}-1}{q-1}, \frac{q^{m}-1}{q-1}, \frac{q^{m-1}-1}{q-1} )$, where $m$ and $q$ are odd and $m\geq3$. Let $(V, f )$ be an orthogonal space of dimension $m+2$. In this design, the points are all isotropic $1$-dimensional subspaces and the blocks are of the form $B(X)=\{Y \in \Pmc \mid f(X,Y)=0\}$, for $X\in \Pmc$. The group $\PGaO_{m+2}(q)$ is an automorphism group of this design and the group $\PGaO_{m+2}(q)$ is its full automorphism group. For $q=3$ and $m=3$, we have the symmetric $(40, 13, 4)$ design with rank $3$ antiflag-transitive and point-primitive automorphism group $\POm_{5}(3)\cong \PSU_{4}(2)$~\cite{a:Braic-2500-nopower,b:Atlas,a:rank3}. Thus the complement of this design with parameters $(40, 27, 18)$ is flag-transitive.
  \item[Lines 7-8] It is shown in  ~\cite[Theorem 3.3]{a:Praeger-45-12-3} that, up to isomorphism, there is only one symmetric $(45, 12, 3)$ design with flag-transitive and point-primitive full automorphism group $\PSU_{4}(2):2$. This design can also be obtained from orthogonal space. Let $(V, f )$ be a non-degenerate orthogonal space of dimension $2m+1$ over $\F_{3}$ with discriminant $(-1)^m$, for $m>1$. The point set $\Pmc$ consists of all anisotropic $1$-dimensional subspaces $X =  \langle x \rangle \leq V$ satisfying $f (x,x)=-1$, and blocks have the form $B(X)=\{ Y\in \Pmc\mid  f (X,Y)=0\}$, for $X \in \Pmc$. Then $\Dmc=(\Pmc, \Bmc)$ is a symmetric with parameters $(3^m(3^m +1)/2,3^{m-1}(3^m -1)/2,3^{m-1}(3^{m-1}-1)/2)$ (see~\cite{ a:Braic-2500-nopower,a:rank3}). The design in these lines obtained when $m=2$ with flag-transitive and point-primitive automorphism group $\POm_{5}(3)\cong \PSU_{4}(2)$~\cite{a:Braic-2500-nopower,b:Atlas,a:rank3}.
\end{description}

Note that Praeger \cite[Theorem 3.3]{a:Praeger-45-12-3} proved that, up to isomorphism, there is the unique flag-transitive and point-primitive symmetric design with parameters $(45,12,3)$ whose full automorphism group is $\PSU_{4}(2):2$. It is worthy here to mention that she also proved that, up to isomorphism, there is exactly one flag-transitive and point-imprimitive symmetric design with this parameters, see \cite[Corollary 1.2]{a:Praeger-45-12-3}.


\section{Preliminaries}\label{sec:pre}

In this section, we state some useful facts in both design theory and group theory. Recall that a group $G$ is called almost simple if $X\unlhd G\leq \Aut(X)$, where $X$ is a nonabelian simple group. If $H$ is a maximal subgroup not containing the socle $X$ of an almost simple group $G$, then $G=HX$, and since we may identify $X$ with $\Inn(X)$, the group of inner automorphisms of $X$, we also conclude that $|H|$ divides $|\Out(X)|\cdot |X\cap H|$. This implies the following elementary and useful fact:

\begin{lemma}\label{lem:New}{\rm \cite[Lemma 2.2]{a:ABD-PSL3}}
Let $G$  be an almost simple group with socle $X$, and let $H$ be maximal in $G$ not containing $X$. Then
\begin{enumerate}[{\rm \quad (a)}]
  \item $G=HX$;
  \item $|H|$ divides $|\Out(X)|\cdot |X\cap H|$.
\end{enumerate}
\end{lemma}

\begin{lemma}\label{lem:Tits}
Suppose that $\Dmc$ is a symmetric $(v,k,\lambda)$ design admitting a flag-transitive and point-primitive almost simple automorphism group $G$ with socle $X$ of Lie type. Suppose also that the point-stabiliser $G_{\alpha}$, not containing $X$, is not a parabolic subgroup of $G$. Then $\gcd(p,v-1)=1$.
\end{lemma}
\begin{proof}
Note that $G_{\alpha}$ is maximal in $G$, then by Tits' Lemma \cite[1.6]{a:tits}, $p$ divides $|G:G_{\alpha}|=v$, and so  $\gcd(p,v-1)=1$.
\end{proof}

\begin{lemma}\label{lem:subdeg}{\rm \cite[3.9]{a:LSS1987}}
If $X$ is a group of Lie type in characteristic $p$, acting on the set
of cosets of a maximal parabolic subgroup, and $X$ is not $\PSL_n(q)$, $\POm_{2m}^{+}(q)$
(with $m$ odd) and $E_{6}(q)$, then there is a unique subdegree which is a power of $p$.
\end{lemma}

\begin{lemma}\label{lem:six}{\rm \cite[Lemma 2.1]{a:ABD-PSL2}}
Let $\Dmc$ be a symmetric $(v,k,\lambda)$ design, and let $G$ be a flag-transitive automorphism group of $\Dmc$. If $\alpha$ is a point in $\Pmc$ and $H:=G_{\alpha}$, then
\begin{enumerate}[\rm \quad (a)]
  \item $k(k-1)=\lambda(v-1)$;
  \item $4\lambda(v-1)+1$ is square;
  \item $k\mid |H|$ and $\lambda v<k^2$;
  \item $k\mid \gcd(\lambda(v-1),|H|)$;
  \item $k\mid \lambda d$, for all subdegrees $d$ of $G$.
\end{enumerate}
\end{lemma}

If a group $G$ acts primitively on a set $\Pmc$ and $\alpha\in \Pmc$ (with $|\Pmc|\geq 2$), then the point-stabiliser $G_{\alpha}$ is maximal in $G$ \cite[Corollary 1.5A ]{b:Dixon}. Therefore, in our study, we need a list of all maximal subgroups of almost simple group $G$ with socle $X:=\PSU_{4}(q)$. Note that if $H$ is a maximal subgroup of $G$, then $H_{0}:=H\cap X$ is not necessarily maximal in $X$ in which case $H$ is called a \emph{novelty}. By~\cite[Tables 8.10 and 8.11]{b:BHR-Max-Low}, the complete list of maximal subgroups of an almost simple group $G$ with socle $\PSU_{4}(q)$ are known, and in this case, there arise only three novelties.

\begin{lemma}\label{lem:maxes}
Let $G$ be a group such that $\PSU_{4}(q)\lhd G \leq \Aut(X)$, and let $H$ be a maximal subgroup of $G$ not containing $X=\PSU_{4}(q)$ and $d=\gcd(4,q+1)$. Then $X\cap H$ is (isomorphic to) one of the subgroups listed in {\rm Table~\ref{tbl:maxes}}.
\end{lemma}
\begin{proof}
The maximal subgroups $H$ of $G$ can be read off from~\cite[Tables 8.10 and 8.11]{b:BHR-Max-Low}.
\end{proof}
\begin{table}[h]
  \scriptsize
  \centering
  \caption{Maxiamal subgroups $H$ of almost simple groups with socle $X=\PSU_{4}(q)$.}\label{tbl:maxes}
\begin{tabular}{cp{4cm}p{5cm}}
  \noalign{\smallskip}\hline\noalign{\smallskip}
   Line& $H\cap X$ & Comments   \\
  \hline\noalign{\smallskip}
  $1$ & $^{\hat{}}E_{q}^{1+4}:\SU_{2}(q):(q^{2}-1)$ &
  \\
  $2$ & $^{\hat{}}E_{q}^{4}:\SL_{2}(q^{2}):(q-1)$ &   \\
  $3$ & $^{\hat{}}\GU_{3}(q)$ &
  \\
  $4$ & $^{\hat{}}(q+1)^{3}:S_{4}$ & novelty if $q=3$
  \\
  $5$ & $^{\hat{}}\SU_{2}(q)^{2}:(q+1)\cdot 2$ & $q\geq 3$
  \\
  $6$ & $^{\hat{}}\SL_{2}(q^{2})\cdot (q-1)\cdot 2$ & $q\geq 4$, novelty if $q=3$
  \\
  $7$ & $^{\hat{}}\SU_{4}(q_{0})$ & $q=q_{0}^r$ and $r$ odd prime
  \\
  $8$ & $^{\hat{}}\SP_{4}(q)\cdot \gcd(2,q+1)$ &
  \\
  $9$ & $^{\hat{}}{\SO_{4}^{+}}(q)\cdot d$ & $q\geq 5$ odd
  \\
  $10$ & $^{\hat{}}{\SO_{4}^{-}}( q)\cdot d$ & $q$ odd
  \\
  $11$ & $^{\hat{}}(4\circ 2^{1+4})^{\cdot}S_{6}$ & $p=q\equiv 7 \mod 8$ \\
  $12$ & $^{\hat{}}(4\circ 2^{1+4})\cdot A_{6}$  & $p=q\equiv 3 \mod 8$ \\
  $13$ & $^{\hat{}}d\circ 2^{\cdot}\PSL_{2}(7)$ & novelty, $q=p\equiv 3, 5, 6 \mod 7$, $q\neq 3$ \\
  $14$ & $^{\hat{}}d\circ 2^{\cdot}A_{7}$ & $q=p\equiv 3, 5, 6 \mod 7$ \\
  $15$ & $^{\hat{}}{4_{2}^{\cdot}}\PSL_{3}(4)$ &  $q=3$\\
  $16$ & $^{\hat{}}d\circ 2^{\cdot}\PSU_{4}(2)$ & $q=p\equiv 5 \mod 6$ \\
  \hline\noalign{\smallskip}
  \multicolumn{3}{l}{Note: $d:=\gcd(4,q+1)$}
\end{tabular}

\end{table}

\section{\bf Proof of the main result}\label{sec:proof}

In this section, suppose that $\Dmc$ is a nontrivial symmetric $(v, k, \lambda)$ design and  $G$ is an almost simple automorphism group $G$ with simple socle $X:=\PSU_{4}(q)$, where  $q = p^a$ with $p$ prime, that is to say, $X\lhd G \leq \Aut(X)$. Suppose also that $V=\F_{q}^{4}$ is the underlying vector space of $X$ over the finite field $\F_{q}$.

Let now $G$ be a flag-transitive and point-primitive automorphism group of $\Dmc$. Then the point-stabiliser $H := G_{\alpha}$ is maximal in $G$ \cite[Corollary 1.5A]{b:Dixon}. Set $H_{0}:= H\cap X$. Then by Lemma~\ref{lem:maxes}, the subgroup $H_{0} $ is (isomorphic to) one of the subgroups as in Table~\ref{tbl:maxes}. Moreover, by Lemma~\ref{lem:New},
\begin{align}
   v=\frac{|X|}{|H_{0}|}=\frac{q^6(q^2-1)(q^3+1)(q^4-1)}{\gcd(4, q+1)\cdot |H_{0}|},\label{eq:v}
\end{align}
Note that $|\Out(X)|=2a\cdot \gcd(4,q+1)$. Therefore, by Lemmas~\ref{lem:New}(b) and~\ref{lem:six}(c),
\begin{align}\label{eq:k-out}
  k \mid 2a\cdot \gcd(4,q+1) \cdot |H_{0}|.
\end{align}

We now consider all possibilities for the subgroup $H_{0}$ as in Table~\ref{tbl:maxes}, and prove that the only possible cases are those have been listed in Table~\ref{tbl:main}.

\begin{lemma}\label{lem:su2}
If $H_{0} =$ $^{\hat{}}E_{q}^{1+4}:\SU_{2}(q):(q^{2}-1)$, then $q=2$ and $(v, k, \lambda)=(45, 12, 3)$.
\end{lemma}
\begin{proof}

In this case, $|H_0|=q^6(q^2-1)^2/\gcd(4,q+1)$, and so  by~\eqref{eq:v}, we have that $v=q^5+q^3+q^2+1$. Then by Lemma \ref{lem:six}(a), $k$ divides $\lambda(v-1)=\lambda q^2(q^3+q+1)$. It follows from Lemma \ref{lem:subdeg} that $G$ has a subdegree $p^b$ of prime power $p$, and so by Lemma~\ref{lem:six}(e),  we conclude that $k$ divides $\lambda p^b$. Hence $k$ divides $\lambda \gcd(p^b,v-1)=\lambda\gcd(p^{b},q^2(q^3+q+1))$, and since $p^b$ divides $q^6$, it follows  that $k$ divides $\lambda q^2$. Let now $m$ be a positive integer such that $mk = \lambda q^2$. Since $\lambda<k$, we have that
\begin{align}\label{eq:case-1-m}
  m<q^2.
\end{align}
By Lemma~\ref{lem:six}(a), $k(k-1)=\lambda(v-1)$, and so
\begin{align*}
  \frac{\lambda q^2}{m}(k-1) = \lambda(q^5+q^3+q^2).
\end{align*}
Thus,
\begin{align}
  k= m(q^3+q+1)+1 \quad \text{and} \quad
\lambda =m^2q+\frac{m^2(q+1)+m}{q^2}.\label{eq:case1-lam}
\end{align}
Since $\lambda$ is integer, (\ref{eq:case1-lam}) implies that
\begin{align}\label{eq:case1-2}
q^2\mid m^2(q+1)+m.
\end{align}
It is easy to know that $\gcd(q^{2},m)=1$, and so $q^{2}$ divides $m(q+1)+1$. Let $n$ be a positive integer such that $m(q+1)+1=nq^{2}$. Then
\begin{align*}
  m=\frac{nq^2-1}{q+1}=n(q-1)+\frac{n-1}{q+1}.
\end{align*}

If $n\neq 1$, then $q+1$ would divide $n-1$, and so $n\geq  q+2$. Note by~\eqref{eq:case-1-m} that $nq^2=m(q+1)+1<q^2(q+1)+1$ which implies that $n\leq q+1 $, which is a contradiction. Therefore, $n=1$, and hence  $m=q-1$. It follows from~\eqref{eq:case1-lam} that $k=q^2(q^2-q+1)$. By~\eqref{eq:k-out}, $k$ divides $2aq^5(q-1)^2$. Therefore, $q^2-q+1$ must divide $2a(q^2-1)^2$. Since $\gcd(q^2-q+1, q-1)=1$ and $\gcd(q^2-q+1, q^2+2q+1)$ divides $3$, $q^2-q+1$ must divide $6a$.
This holds only when $q=2$ in which case $v=45$, $k=12$ and $\lambda=3$.  By~\cite[Theorem 3.3]{a:Praeger-45-12-3}, this design is unique (up to isomorphism) with full automorphism group $\PSU_{4}(2):2$.
\end{proof}

\begin{lemma}\label{lem:sl2}
The subgroup $H_{0}$ cannot be $^{\hat{}}E_{q}^{4}:\SL_{2}(q^{2}):(q-1)$.
\end{lemma}
\begin{proof}
Here $|H_0|=d^{-1}q^6(q^4-1)(q-1)$, where $d=\gcd(4,q+1)$. According to ~\eqref{eq:v}, we have that $v=q^4+q^3+q+1$. Note by  Lemma \ref{lem:six}(a) that $k$ divides $\lambda(v-1)=\lambda q(q^3+q^2+1)$. Moreover, by Lemma \ref{lem:subdeg} and Lemma~\ref{lem:six}(e), $k$ divides $\lambda p^b$, where $p^b$ is a prime power subdegree of $G$. Therefore $k$ divides $\lambda \gcd(p^b,v-1)=\lambda\gcd(p^{b},q(q^3+q^2+1))$, and since $p^b$ divides $q^6$, it follows that $k$ divides $\lambda q$. If $m$ is a positive integer such that $mk = \lambda q$, then since $\lambda<k$, we have that
\begin{align}\label{eq:case-2-m}
  m<q.
\end{align}
By Lemma~\ref{lem:six}(a), $k(k-1)=\lambda(v-1)$, and so
\begin{align*}
  \frac{\lambda q}{m}(k-1) = \lambda(q^4+q^3+q).
\end{align*}
Thus,
\begin{align}
  k= m(q^3+q^2+1)+1 \quad \text{and} \quad
\lambda = m^2(q^2+q)+\frac{m^2+m}{q}.\label{eq:case2-lam}
\end{align}
It follows from \eqref{eq:case2-lam} that $q\mid m^2+m$.  It is easy to know that $\gcd(q,m)=1$, and so $q$ divides $m+1$. By \eqref{eq:case-2-m}, we conclude that $m=q-1$, and hence $k=q(q^3-q+1)$ and $\lambda=q^3-1$ by \eqref{eq:case2-lam}. Note by Lemma~\ref{lem:six}(c) that $k$ divides $8aq^6(q^4-1)(q-1)$. Then $q^3-q+1$ must divide $8a(q^4-1)(q-1)$. Since $\gcd(q^3-q+1, q-1)=1$,  $q^3-q+1$ must divide $8a(q^3+q^2+q+1)$. Therefore $q^3-q+1$ divides $8a(q^2+2q)$, which is impossible.
\end{proof}


\begin{lemma}\label{lem:GU3}
If $H_{0}$ is $^{\hat{}}\GU_{3}(q)$, then $q=2$ and $(v, k, \lambda)=(40, 27, 18)$.
\end{lemma}
\begin{proof}
Let $\{u_1,u_2,u_3,u_4\}$ be a canonical basis for the underlying unitary space $V$. In this case, $H=G_U$, where $U$ is a $1$-dimensional non-degenerate subspace, say $U=\langle u_1\rangle $. Then $|H_0|=q^3(q^2-1)(q^3+1)(q+1)/\gcd(4,q+1)$ which implies by \eqref{eq:v} that $v=q^{3}(q-1)(q^{2}+1)$. Let now $W:=\langle u_{1}, u_{2} \rangle$. Then $G$ has a subdegree $|G_U:G_{U,W}|$ dividing $(q+1)(q^{3}+1)$ (see \cite[p. 549]{a:reg-classical} and \cite[p. 336]{a:Saxl2002}). Therefore Lemma~\ref{lem:six}(d) implies that $k$ must divide $\lambda (q+1)(q^{3}+1)$. On the other hand, $k$ divides $\lambda(v-1)=\lambda(q^{2}-q+1)(q^{4}-q-1)$. Therefore, $k$ divides $\lambda(q^{2}-q+1)$, and so $mk=\lambda(q^{2}-q+1)$, for some positive integer $m$. Then
\begin{align}\label{eq:GU3-m}
  m<q^{2}-q+1.
\end{align}
By Lemma~\ref{lem:six}(a), we have that  $k(k-1)=\lambda(v-1)$, and so
\begin{align}
  k= m(q^4-q-1)+1.\label{eq:GU3-k}
\end{align}

We first show that $q^2$ does not divide $k$. If $q^2$ would divide $k$, then by~\eqref{eq:GU3-k}, $q^2$ should divide $m(q+1)-1$. Let now $n$ be a positive integer such that $m(q+1)-1=nq^{2}$. Then
\begin{align*}
  m=\frac{nq^2+1}{q+1}=n(q-1)+\frac{n+1}{q+1}.
\end{align*}
Therefore, $q+1$ must divide $n+1$, and so $n\geq q$. Note by \eqref{eq:GU3-m} that $nq^2=m(q+1)-1<(q^2-q+1)(q+1)-1=q^3$. Thus $n \leq q-1$, which is a contradiction. Therefore, $q^2$ does not divide $k$.\smallskip

Note by Lemma~\ref{lem:New}(b) that $k$ divides $2ag(q)$, where $g(q)=q^3(q+1)(q^2-1)(q^3+1)$.
Since $k$ is not a multiple of $q^2$, we must have $k\mid 2ag_{1}(q)$, where $g_{1}(q)=g(q)/q=q^2(q+1)(q^2-1)(q^3+1)$. Then, by~\eqref{eq:GU3-k}, we must have
\begin{align}\label{eq:GU3-q2}
m(q^4-q-1)+1\mid 2ag_{1}(q).
\end{align}
Let now $d(q)=q^3+q^2-4q-3$ and $h(q)=q^4+q^3-q^2+q+3$. Then $2ah(q)[m(q^4-q-1)+1]-2mag_{1}(q)=2mad(q)+2ah(q)$, and so~\eqref{eq:GU3-q2} implies that $m(q^4-q-1)+1$ divides $2mad(q)+2ah(q)$. Thus $m[q^4-q-1-2ad(q)]+1\leq 2ah(q)$. Note that $m\leq 2ah(q)/[q^4-q-1-2ad(q)]\leq 33$, for $q\neq 4$. Therefore, \eqref{eq:GU3-m} implies that $m\leq \min\{33, q^2-q+1\}$, for all $q=p^a$.\smallskip

We now show that $q$ does not divide $k$. If $q$ would divide $k$, then by~\eqref{eq:GU3-k}, $q$ should divide $m-1$. As $m\leq \min\{33, q^2-q+1\}$, it follows that $q\leq 32$. Therefore,
\begin{align}\label{eq:GU3-ap-1}
  \begin{array}{llll}
    p =2, & \quad a\leq 5; \\
    p =3, & \quad a\leq 3; \\
    p =5, & \quad a\leq 2; \\
    p =7,  11, 13, 17, 19, 23, 29, 31,& \quad a= 1.
 \end{array}
\end{align}
For the pairs $(p,a)$ as in \eqref{eq:GU3-ap-1}, since $m\leq \min\{33, q^2-q+1\}$ and $m\equiv 1 \mod q$, the parameter $k= m(q^4-q-1)+1$ does not divide $2ag_{1}(q)$, which is a contradiction. Therefore, $k$ is not a multiple of $q$. Again applying Lemmas~\ref{lem:New}(b) and \eqref{eq:GU3-k}, we have that
\begin{align}\label{eq:GU3-1}
  m(q^4-q-1)+1\mid 2ag_{2}(q),
\end{align}
where $g_{2}(q)=g(q)/q^2=q(q+1)(q^2-1)(q^3+1)$. If $d_{1}(q)=3q^3-q^2-q+1$ and $h_{1}(q)=q^3+q^2-q+1$, then $2ah_{1}(q)[m(q^4-q-1)+1]-2mag_{2}(q)=2a[md_{1}(q)-h_{1}(q)]$.
It follows from \eqref{eq:GU3-1} that $m(q^4-q-1)+1$ divides $2a[md(q)-h(q)]$, and so $q^4-q-1<2a[|d(q)|+|h(q)|]$. This inequality holds only for $q \in \{2, 3, 4, 5, 7, 8, 9, 16, 32\}$. For these values of $q$, as $k= m(q^4-q-1)+1$ divides $2ag_{2}(q)$, for $m\leq \min\{33, q^2-q+1\}$, we conclude that $q=2$ in which case $v=40$, $k=27$ and $\lambda=18$. It follows from~\cite{a:Braic-2500-nopower,a:rank3} that the design $\Dmc$ is the complement of $\PG(3,3)$ with parameters $(40,13,4)$ and  flag-transitive and point-primitive automorphism group $\PSU_{4}(2)$ or~$\PSU_{4}(2):2$.
\end{proof}


\begin{lemma}\label{lem:q+1:s4}
If $H_{0}$ is $^{\hat{}}(q+1)^3:S_{4}$, then $q=2$ and $(v, k, \lambda)=(40, 27, 18)$.
\end{lemma}
\begin{proof}
In this case, $|H_0|=24d^{-1}(q+1)^3$, where $d=\gcd(4,q+1)$. Then by~\eqref{eq:v}, we have $v=q^6(q-1)^2(q^2+1)(q^2-q+1)/24 $, and since $|\Out (X)|= 2a\cdot\gcd(4, q+1)$, it follows from (\ref{eq:k-out}) that $k$ divides $48a(q+1)^3$. By \cite{a:reg-classical,a:Zhou-lam3-classical} and Lemma~\ref{lem:six}(c), we may assume that $\lambda$ is at least $4$, and so
\begin{align*}
  \frac{q^6(q-1)^2(q^2+1)(q^2-q+1)}{6}\leq \lambda v< k^2 \leq 48^2a^2(q+1)^6.
\end{align*}
This implies that $q^6(q-1)^2(q^2+1)(q^2-q+1)<13824 a^2 (q+1)^6$. Thus
\begin{align*}
  \frac{q^6(q-1)^2(q^2+1)(q^2-q+1)}{(q+1)^6}<13824 a^2 .
\end{align*}
This inequality is true only when $q \in \{2,3,4,5,8\}$. Since $k$ is a divisor of $48a(q+1)^3$, for each such $q=p^a$, the possible values of $k$ and $v$ are listed in Table~\ref{tbl:(q+1)3:S4}.
\begin{table}[h]
\centering\scriptsize
  \caption{Possible value for $k$ and $v$ when $q \in \{2,3,4,5,8\}$.}
  \label{tbl:(q+1)3:S4}
  \begin{tabular}{clllll}
   \noalign{\smallskip}\hline\noalign{\smallskip}
    $q$ & $2$ & $3$ & $4$ & $5$ & $8$ \\
    \hline\noalign{\smallskip}
    $v$ & $40$ & $8505$ & $339456$ & $5687500$ & $1982955520$ \\
    $k$ divides & $1296$ & $3072$ & $12000$ & $10368$ & $104976$\\
   \hline\noalign{\smallskip}
  \end{tabular}
\end{table}

The only possible parameters $(v,k,\lambda)$ satisfying  $\lambda < k<v-1$ and $\lambda(v-1)=k(k-1)$ is $(v, k, \lambda)=(40, 27, 18)$ when $q=2$. By~\cite{a:Braic-2500-nopower,a:rank3}, the design $\Dmc$ is the Higman design with parameters $(40,13,4)$ and flag-transitive and point-primitive automorphism group $\PSU_{4}(2)$ or $\PSU_{4}(2):2$.
\end{proof}

\begin{lemma}\label{lem:SU22:(q+1)}
The subgroup $H_{0}$ cannot be  $^{\hat{}}\SU_{2}(q)^{2}:(q+1)\cdot 2$, for $q\geq 3$.
\end{lemma}
\begin{proof}
In this case, $H$ preserves a decomposition $V = V_{1}\bigoplus V_{2}$ of nonsingular subspaces $V_{1}=\langle u_1,u_2\rangle$ and $V_{2}=\langle u_3,u_4\rangle$.  Take  the partition $y:=\{\langle  u_{1},u_{3} \rangle,\langle u_{2}, u_{4} \rangle\}$. Then the subdegree $|H:H_y|$ of $G$ divides $2(q^{2}-1)^2$ (see \cite[p. 550]{a:reg-classical} and \cite[pp. 336-337]{a:Saxl2002}). Thus by Lemma~\ref{lem:six}(e), we conclude that $k$ divides $2\lambda (q^{2}-1)^2$. Note in this case that $|H_0|=2d^{-1}q^2(q^2-1)^2(q-1)$, where $d=\gcd(4,q+1)$. By~\eqref{eq:v}, we have that $v=q^4(q^2-q+1)(q^2+1)/2$.

Since $2(v-1)=(q+1)(q^7+2q^5+q^4+2q^3+2q^2+2q+2)+4$, we conclude that $\gcd(v-1, q+1)$ divides $2$. Note also that $q-1$ divides $v-1$. Therefore, $\gcd(v-1,2(q^{2}-1)^2)$ divides $8(q-1)^2$. Since $k$ divides $\lambda\gcd(v-1,2(q^{2}-1)^2)$, we conclude that $k$  divides $\lambda f(q)$, where $f(q)=8(q-1)^2$. Thus $mk = \lambda f(q)$, for some positive integer $m$. Since $k(k-1)=\lambda(v-1)$ and $\lambda<k$, it follows that
\begin{align}\label{eq:case5-1}
  k= \frac{m(v-1)}{f(q)}+1,
\end{align}
where $f(q)=8(q-1)^2$ and
\begin{align}\label{eq:case2-su2-m}
  m<8(q-1)^2.
\end{align}
Note by (\ref{eq:k-out}) that $k\mid 4ag(q)$, where $g(q)=q^2(q-1)^2(q+1)^3$. Then, by~\eqref{eq:case5-1}, we must have
\begin{align}\label{eq:case5-2}
m(v-1)+f(q)\mid 4af(q)g(q).
\end{align}
Let now $d(q)=80q^7-64q^6-32q^5+48q^4+16q^3-16q^2-32q$ and $h(q)=32q$. Then
\begin{align*}
  ah(q)[m(v-1)+f(q)]-4am f(q)g(q)= mad(q)+af(q)h(q).
\end{align*}
Therefore, (\ref{eq:case5-2}) implies that
\begin{align*}
 m(v-1)+f(q)&\leq a(md(q)+f(q)h(q))
\end{align*}
So $q^4(q^2-q+1)(q^2+1)<2a[d(q)+f(q)h(q)]$. Since also $d(q)+f(q)h(q)< 80q^{7}$, for all $q\geq 2$. Therefore, $(q^2-q+1)(q^2+1)<160 a q^3$. This inequality holds only for pairs $(p,a)$ as in Table~\ref{tbl:case5-ap} below:
\begin{table}[h]
\centering\scriptsize
   \caption{Some parameters for Lemma~\ref{lem:SU22:(q+1)}\label{tbl:case5-ap}}
  \begin{tabular}{cccccccccc} \noalign{\smallskip}\hline\noalign{\smallskip}
  $p$ &
  $2$ &
  $3$ &
  $5$ &
  $7$ &
  $11$, $13$, $17$ &
  $19$, \ldots, $157$ \\
  \hline\noalign{\smallskip}
  $a\leq $ &
  $10$ &
  $6$ &
  $4$ &
  $3$ &
  $2$ &
  $1$ \\
  \hline\noalign{\smallskip}
\end{tabular}
\end{table}

For these values of $q=p^a$, and the parameter $m$ as in \eqref{eq:case2-su2-m}, there is no parameter $k$ satisfying \eqref{eq:case5-1} for which the fraction $k(k-1)/(v-1)$ is a positive integer, which is a contradiction.
\end{proof}


\begin{lemma}\label{lem:SL2:(q-1).2}
The subgroup $H_{0}$ cannot be $^{\hat{}}\SL_{2}(q^{2}):(q-1)\cdot 2$.
\end{lemma}
\begin{proof}
Let $\{e_1,e_2,f_1,f_2\}$ be a standard basis for underlying unitary space $V$. In this case, $H$ preserves a partition $V = V_{1}\bigoplus V_{2}$ of totally singular subspaces $V_{1}$ and $V_{2}$ of dimension $2$, say $V_1=\langle e_1,e_2\rangle$ and $V_2=\langle f_1,f_2\rangle$.  Let now $y=\{\langle  e_{1}, f_{4} \rangle,\langle e_{2}, f_{3} \rangle\}$. Then the subdegree $|H:H_{y}|$ of $G$ divides $2(q^{4}-1)$ (see \cite[p. 550]{a:reg-classical} and \cite[pp. 336-337]{a:Saxl2002}). Thus by Lemma~\ref{lem:six}(e), we conclude that $k$ divides $2\lambda (q^{4}-1)$. Here $|H_0|=2q^2(q^4-1)(q+1)/\gcd(4,q+1)$, and so \eqref{eq:v} implies that $v=q^4(q^3+1)(q+1)/2$.

Note that $2(v-1)=(q-1)(q^7+2q^6+2q^5+3q^4+4q^3+4q^2+4q+4)+2$ and $2(v-1)=(q+1)(q^7+q^4)-2$. Then $v-1$ is coprime to $q^{2}-1$. Note also that $q^{2}+1$ divides $2(v-1)$. Thus $\gcd(v-1,2(q^{4}-1))$ divides $2(q^{2}+1)$. Since $k$ divides $\lambda \gcd(v-1,  2(q^{4}-1))$, it follows that $k$ divides $\lambda f(q)$, where $f(q)=2(q^2+1)$, and hence $mk = \lambda f(q)$, for some positive integer $m$. Therefore,
\begin{align}\label{eq:case6-1}
  k= \frac{m(v-1)}{f(q)}+1,
\end{align}
where
\begin{align}\label{eq:case2-sl2-m}
  m<2(q^2+1).
\end{align}

Note by (\ref{eq:k-out}) that $k\mid 4ag(q)$, where $g(q)=q^2(q^4-1)(q-1)$. Then, by~\eqref{eq:case6-1}, we must have
\begin{align}\label{eq:case6-2}
m(v-1)+f(q)\mid 4af(q)g(q).
\end{align}
Let now $d(q)=48q^7-32q^6+48q^4-16q^3+16q^2+32q-64$ and $h(q)=32(q-2)$. Then
\begin{align*}
  4am f(q)g(q)-ah(q)[m(v-1)+f(q)]= am d(q)-a f(q)h(q).
\end{align*}
Therefore, \eqref{eq:case6-2} implies that
 $m(v-1)+f(q)\leq |am d(q)-a f(q)h(q)|< am [d(q)+f(q)h(q)]$.
So $q^4(q^3+1)(q+1)<2a[d(q)+f(q)h(q)]$. Since $d(q)+f(q)h(q)< 48q^{7}$, for all $q\geq 2$, $(q^3+1)(q+1)< 96a q^3$. This inequality holds only for pairs $(p,a)$ as in Table~\ref{tbl:case6-ap} below:
\begin{table}[h]
\centering\scriptsize
   \caption{Some parameters for Lemma~\ref{lem:SL2:(q-1).2}\label{tbl:case6-ap}}
  \begin{tabular}{c|ccccccccc}
   \noalign{\smallskip}\hline\noalign{\smallskip}
  $p$ &
  $2$ &
  $3$ &
  $5$ &
  $7$, $11$, $13$ &
  $17$, \ldots, $89$\\
   \hline\noalign{\smallskip}
  $a\leq $ &
  $9$ &
  $5$ &
  $3$ &
  $2$ &
  $1$ \\
   \noalign{\smallskip}\hline\noalign{\smallskip}
\end{tabular}
\end{table}

The only value of $q=p^a$ satisfying \eqref{eq:case6-1}  when $m$ is as in \eqref{eq:case2-sl2-m} for which the fraction $k(k-1)/(v-1)$  is a positive integer is $q=4$ when $m=2$. In which case, we obtain the parameters $(v, k, \lambda)=(41600, 2448, 144)$ with $X=\PSU_{4}(4)$. In what follows, we make use of the software \GAP \  \cite{GAP4} and show that such a design never exists.

Let $G$ be one of the groups $X$, $X:2$ or $X:4$, and  $H$ is $H_0$, $H_0\cdot 2$ or $H_0\cdot 4$, respectively.
We note that the group $G$ has one conjugacy class of subgroup containing $H_0$. We use the command \verb|AtlasGroup("U4(4)")| to define the group $X$,  and then we find all subgroups $G$ of $\Aut(X)$ containing $X$. Since the maximal subgroups $H$ of $G$ is not available in \GAP, we need to construct $H$ as a subgroup of $G$.  We first define the semidirect product $T:=\PSL_2(2^4)\cdot 3$ and then we embed this group $T$ into $G$ as a subgroup via command \verb|IsomorphicSubgroups(G,T)|. For each group $G$, there is only one such isomorphic subgroup $K$ in $G$, and then by \verb|IntermediateSubgroups(G,K)|, we find the overgroups of $K$. Now we can choose those subgroups $H$ of index $41600$. Then we define the right coset action of $G$ on the set $\Pmc:=\mathcal{R}_H$ of right cosets of $H$ in $G$, and so we can view $G$ and $H$ as  subgroups of $S_{41600}$ by taking image of the permutation representation of the right coset action. We now obtain the $H$-orbits on $\Pmc$ and the subdegrees of $G$ which are listed in Table~\ref{tbl:subdegs}. Since $G$ is flag-transitive, each $H$-orbit of size $2448$ (if there exists) would be a possible base block $B$ for $\Dmc$. At this stage, we obtain two base blocks for each group $G$, see Table~\ref{tbl:subdegs}. Although, the command \verb|BlockDesign( 41600, [B], G )| returns true for the obtained base blocks, these designs are not symmetric as $|B^x\cap B|\neq 144$, for some $x\in G$.
\end{proof}

\begin{table}[h]
\centering\scriptsize
   \caption{Some subdegrees of almost simple group $G$ with socle $\PSU_{4}(4)$.}\label{tbl:subdegs}
  \begin{tabular}{lp{11cm}}
   \noalign{\smallskip}\hline\noalign{\smallskip}
  \multicolumn{1}{c}{$G$} &
  \multicolumn{1}{c}{Subdegrees} \\
  \hline\noalign{\smallskip}
  $\PSU_{4}(4)$&
  $1$, $102$, $136$, $153$, $204$, $408$, $816$, $816$, $1224$, $1632$, $2040$, $2040$, $2448$, $2448$, $3060$, $ 4080$, $4080$, $4896$, $4896$, $6120$ \\
  $\PSU_{4}(4):2$ &
   $1$, $102$, $136$, $153$, $204$, $408$, $816$, $816$, $1224$, $1632$, $2040$, $2040$, $2448$, $2448$, $3060$, $ 4080$, $4080$, $
4896$, $4896$, $6120$
   \\
  $\PSU_{4}(4):4$&
   $1$, $102$, $136$, $153$, $204$, $408$, $1224$, $1632$, $1632$, $ 2448$, $2448$, $3060$, $4080$, $6120$, $8160$, $9792$\\
   \hline\noalign{\smallskip}
\end{tabular}
\end{table}

\begin{lemma}\label{lem:D-last}
The subgroup $H_{0}$ cannot be $^{\hat{}}\SU_{4}(q_{0})$, where $q=q_{0}^r$ and $r$ odd prime.
\end{lemma}
\begin{proof}
In this case,  $|H_{0}|=q_{0}^{6}(q_{0}^{4}-1)(q_{0}^{3}+1)(q_{0}^{2}-1)/\gcd(4,q_0^r+1)$. It follows from~\eqref{eq:v} that
\begin{align}\label{eq:lem-SU4-v}
  v= \frac{q_{0}^{6r}(q_{0}^{4r}-1)(q_{0}^{3r}+1)(q_{0}^{2r}-1)}{q_{0}^{6}(q_{0}^{4}-1)(q_{0}^{3}+1)(q_{0}^{2}-1)}
\end{align}
Note by (\ref{eq:k-out}) that $k$ divides $2aq_{0}^{6}(q_{0}^{4}-1)(q_{0}^{3}+1)(q_{0}^{2}-1)$. We may assume that $\lambda\geq 4$  by  \cite{a:reg-classical,a:Zhou-lam3-classical}. Moreover, $a^2\leq q_{0}^r$ as $q=q_{0}^{r}$. Since $\lambda v<k^2$ by Lemma~\ref{lem:six}(b), we must have
\begin{align*}
  \frac{4q_{0}^{6r}(q_{0}^{4r}-1)(q_{0}^{3r}+1)(q_{0}^{2r}-1)}{q_{0}^{6}(q_{0}^{4}-1)(q_{0}^{3}+1)(q_{0}^{2}-1)}
\leq \lambda v< k^2
&\leq4a^2\cdot q_{0}^{12}(q_{0}^{4}-1)^2(q_{0}^{3}+1)^2(q_{0}^{2}-1)^2 \\
  &\leq 4\cdot q_{0}^{12+r}(q_{0}^{4}-1)^2(q_{0}^{3}+1)^2(q_{0}^{2}-1)^2
\end{align*}
and hence
\begin{align*}
q_{0}^{6r}(q_{0}^{4r}-1)(q_{0}^{3r}+1)(q_{0}^{2r}-1)<  q_{0}^{18+r}(q_{0}^{4}-1)^3(q_{0}^{3}+1)^3(q_{0}^{2}-1)^3.
\end{align*}
Note that $q_{0}^{15r-1}\leq q_{0}^{6r}(q_{0}^{4r}-1)(q_{0}^{3r}+1)(q_{0}^{2r}-1)$ and $q_{0}^{18+r}(q_{0}^{4}-1)^3(q_{0}^{3}+1)^3(q_{0}^{2}-1)^3\leq q_{0}^{45+r}$. Then $q_{0}^{15r-1}<  q_{0}^{45+r}$, and this implies that $r =2$ or $3$. Since $r$ is odd, we must have $r=3$. Therefore,
\begin{align}\label{lem:case-7-r2-v}
  v=\frac{q_{0}^{12}(q_{0}^{12}-1)(q_{0}^{9}+1)(q_{0}^{6}-1)}{(q_{0}^{4}-1)(q_{0}^{3}+1)(q_{0}^{2}-1)}.
\end{align}
By (\ref{eq:k-out}), $k$ divides $2aq_{0}^{6}(q_{0}^{4}-1)(q_{0}^{3}+1)(q_{0}^{2}-1)$.  Then by Lemma~\ref{lem:six}(c), we have that
\begin{align*}
  \lambda \cdot\frac{q_{0}^{12}(q_{0}^{12}-1)(q_{0}^{9}+1)(q_{0}^{6}-1)}{(q_{0}^{4}-1)(q_{0}^{3}+1)(q_{0}^{2}-1)}< k^{2}\leq 4a^2q_{0}^{12}(q_{0}^{4}-1)^2(q_{0}^{3}+1)^2(q_{0}^{2}-1)^2.
\end{align*}
Therefore,
\begin{align}\label{eq:lem13-case1-lam}
  \lambda< 4a^2\cdot \frac{(q_{0}^{4}-1)^3(q_{0}^{3}+1)^3(q_{0}^{2}-1)^3}{(q_{0}^{12}-1)(q_{0}^{9}+1)(q_{0}^{6}-1)}\leq 4a^2.
\end{align}
Since $k$ divides $2aq_{0}^{6}(q_{0}^{4}-1)(q_{0}^{3}+1)(q_{0}^{2}-1)$ and $v-1$ is coprime to $q_{0}$ by   Lemma~\ref{lem:Tits}, $k$ must divide $2\lambda a (q_{0}^{4}-1)(q_{0}^{3}+1)(q_{0}^{2}-1)$. Now Lemma~\ref{lem:six}(c) implies that
\begin{align*}
  \lambda\cdot \frac{q_{0}^{12}(q_{0}^{12}-1)(q_{0}^{9}+1)(q_{0}^{6}-1)}{(q_{0}^{4}-1)(q_{0}^{3}+1)(q_{0}^{2}-1)}<k^{2} \leq 4\lambda^2 a^2(q_{0}^{4}-1)^2(q_{0}^{3}+1)^2(q_{0}^{2}-1)^2,
\end{align*}
and so
\begin{align}\label{eq:psl-3}
  \frac{q_{0}^{12}(q_{0}^{12}-1)(q_{0}^{9}+1)(q_{0}^{6}-1)}{(q_{0}^{4}-1)^3(q_{0}^{3}+1)^3(q_{0}^{2}-1)^3}< 4 \lambda a^2.
\end{align}
Since $\lambda\leq 4 a^{2}$ by (\ref{eq:lem13-case1-lam}), it follows that
\begin{align*}
  q_{0}^{12}<16 a^{4}.
\end{align*}
Since also $q_{0}\geq 2$, $2^{6a}<16\cdot a^{4}$, which is impossible.
\end{proof}

\begin{lemma}\label{lem:Sp4}
If $H_{0}$ is $^{\hat{}}\Sp_{4}(q)\cdot \gcd(2,q+1)$, then $q=2$ and $(v, k, \lambda)=(36, 15, 6)$.
\end{lemma}
\begin{proof}
Here $|H_0|=d^{-1}cq^4(q^2-1)(q^4-1)$, where $d=\gcd(4,q+1)$ and $c=\gcd(2,q+1)$. So by~\eqref{eq:v}, we have $v=q^2(q^3+1)/c$, where $c=\gcd(2,q+1)$. It follows from \eqref{eq:k-out} that $k$ divides $2ag(q)$, where $g(q)=q^4(q^2-1)(q^4-1)$. We now consider the following two cases.\smallskip

\noindent \textbf{Case 1:}
Let $q$ be even. Then $c=\gcd(2,q+1)=1$. If $q=2$, then $v=36$. It follows from~\eqref{eq:k-out} that $k$ divides $1440$. We then easily observe that for each divisor $k$ of $1440$, the fraction $k(k-1)/(v-1)$ is not a positive integer unless $k=15$, in which case $v=36$ and $\lambda=6$.  By~\cite{a:Braic-2500-nopower,a:rank3}, this design is a Menon design with parameters $(36, 15, 6)$ and flag-transitive automorphism group $\PSU_{4}(2)$ or $\PSU_{4}(2):2$.\smallskip

Let now $q\geq 4$. Note that $v-1$ is coprime to $q(q^2-1)$. Moreover, since $v-1=(q^2+1)(q^3-q+1)+q-2$ and $q^2+1=(q-2)(q+2)+5$, it follows that $\gcd(v-1,q^2+1)$ divides $\gcd(5,q-2)$. Therefore, $\gcd(v-1,(q^2-1)(q^4-1))$ divides $\gcd(5, q-2)$, we have that $k$ is a divisor of $\lambda e a $, where $e:=\gcd(5, q-2)$. Then there exists a positive integer $m$ such that  $mk=\lambda e a$. Thus,
\begin{align}
  k= \frac{m(v-1)}{ea}+1,\label{eq:case1-Sp4-k}
\end{align}
where
\begin{align}
  m<ea=\gcd(q-2,5)a.\label{eq:case1-Sp4-m}
\end{align}
We first show that $q$ does not divide $k$. Let $q$ divide $k$. Then~\eqref{eq:case1-Sp4-k} implies that $q$ divides $ea-m$. Thus $q\leq ea-m\leq \gcd(q-2,5)a-1$, which is impossible.
Therefore, $q$ does not divide $k$, and so it follows from Lemma~\ref{lem:New}(b) and~\eqref{eq:case1-Sp4-k} that
\begin{align}\label{eq:case1-Sp4-k-2}
  m(v-1)+ea\mid 2ea^{2}g_{1}(q),
\end{align}
where $g_{1}(q)=g(q)/q^3=q(q^2-1)(q^4-1)$. Let $d(q)=q^4+q^3-q^2-q$ and $h(q)=q^2-1$. Then
\begin{align}
   2ea^{2}h(q)[m(v-1)+ea]-2mea^{2}g_{1}(q)=2ea^{2}[md(q)+eah(q)].
\end{align}
Therefore, by \eqref{eq:case1-Sp4-k-2},  we conclude that  $v-1<2ea^{2}[|d(q)|+ea|h(q)|]$. This inequality holds only when $a\leq 9$. Then for each $q=2^a$ with $a\leq 9$, the possible values of $v$ are listed in Table~\ref{tbl:case1-Sp4-mv} below. By~\eqref{eq:case1-Sp4-m}, we can also find an upper bound for $m$ listed as in the third column of Table~\ref{tbl:case1-Sp4-mv}.
\begin{table}[h]
\centering
\scriptsize
  \caption{Possible value for $m$ and $v$ when $q=2^{a}$ with $1<a\leq 9$.}
  \label{tbl:case1-Sp4-mv}
    \begin{tabular}{lll}
     \noalign{\smallskip}\hline\noalign{\smallskip}
    \multicolumn{1}{c}{$q$} &
    \multicolumn{1}{c}{$v$} &
    \multicolumn{1}{c}{$m<$} \\
    \hline\noalign{\smallskip}
    $4$ & $1040$ &$2$\\
    $8$ & $32832$ &$3$\\
    $16$ & $1048832$ &$4$\\
    $32$ & $33555456$ &$25$\\
    $64$ & $1073745920$ &$6$\\
    $128$ & $34359754752$ &$7$\\
    $256$ & $1099511693312$ &$8$\\
    $512$ & $35184372350976$ &$45$\\
     \noalign{\smallskip}\hline\noalign{\smallskip}
  \end{tabular}
\end{table}

We now obtain by~\eqref{eq:case1-Sp4-k}, the parameter $k$, but for such $k$, we can not find any possible parameter $\lambda$ satisfying  Lemma~\ref{lem:six}(a), which is a contradiction.\smallskip

\noindent \textbf{Case 2:}
Let $q$ be odd. Then $c=\gcd(2,q+1)=2$ and $v=q^2(q^3+1)/2$. Note that $q-1$ divides $2(v-1)=q^2(q^3+1)-2$. Set $w(q):=2(v-1)/(q-1)=q^4+q^3+q^2+2q+2$. Then $q+1$ is coprime to $w(q)$. Moreover, $w(q)=(q-1)(q^3+2q^2+3q+5)+7$, $w(q)=(q^{2}+1)(q^{2}+q)+q+2$ and $q^{2}+1=(q+2)(q-2)+5$. Therefore $\gcd(v-1, 2(q^4-1)(q^2-1))$ divides $\gcd(q+2, 5)\gcd(q-1, 7)(q-1)$, and so Lemmas~\ref{lem:six}(a) and~\ref{lem:Tits} imply that $k$ divides $\lambda as f(q)$, where $f(q)=q-1$ and
\begin{align}
  s= \gcd(q+2, 5)\gcd(q-1, 7).\label{eq:case2-Sp4-s}
\end{align}
Then $mk=\lambda a s f(q)$, for some positive integer $m$, and so \begin{align}
  k= \frac{m(v-1)}{as f(q)}+1,\label{eq:case2-Sp4-k}
\end{align}
where $f(q)=q-1$, $s= \gcd(q+2, 5)\gcd(q-1, 7)$ and
\begin{align}\label{eq:case2-Sp4-m}
  m< as (q-1).
\end{align}
As in Case 1, we first show that $q$ does not divide $k$. Assume the contrary. Then~\eqref{eq:case2-Sp4-k} implies that $q\mid m+as$,
and so $nq=m+sa$, for some positive integer $n$. Thus
 \begin{align}\label{eq:case2-Sp4-m1}
 m=nq-as.
\end{align}
 Since $m<sa (q-1)$, we have that
\begin{align}\label{eq:case2-n}
  n<as .
\end{align}
Since also $mk=\lambda as f(q)$ and $k(k-1)=\lambda(v-1)$, it follows that
\begin{align*}
  2a^2s^2 \lambda = m^{2}(q^3+2q^2+3q+5)+\frac{7m^2+2mas}{q-1},
\end{align*}
and so $q-1$ divides $7m^2+2mas$. Therefore,
by~\eqref{eq:case2-Sp4-m1}, we conclude that
\begin{align}\label{eq:case2-Sp4-q-1-m-2}
 q-1\mid 7(n^2q^2-2nasq+a^2s^2)+2(nasq-a^2s^2).
\end{align}
As
\begin{multline*}
  7(n^2q^2-2nasq+a^2s^2)+2(nasq-a^2s^2)= \\
  7n^2(q^2-1)-12nas(q-1)+7n^2-12nas+5a^2s^2,
\end{multline*}
$q-1$ must divide $7n^2-12nas+5a^2s^2$. Since now $n<as$, we conclude that $7n^2-12nas+5a^2s^2>0$, and so  $q-1\leq 7n^2-12nas+5a^2s^2$. Moreover,  $7n^2-12nas<0$. Therefore,
\begin{align}\label{eq:case2-Sp4-q-1-n}
q-1\leq 5a^2s^2.
\end{align}
If $a>1$, then the inequality~\eqref{eq:case2-Sp4-q-1-n} holds only for the pairs $(p,a)$ as  below:
\begin{align}\label{eq:sp4-case2-ap-1}
  \begin{array}{llll}
    p =3, & \quad a=2,3,4,5,6; \\
    p =7, 11, 37, & \quad a=3; \\
    p =13, 29, & \quad a=2.
 \end{array}
\end{align}
Note that $n<as$ and $m=nq-as$, for the values of $(p,a)$ as in~\eqref{eq:sp4-case2-ap-1}, we can find the parameter $k$ from \eqref{eq:case2-Sp4-k}, and hence we easily observe that for these values $k$, the fraction $k(k-1)/(v-1)$ is not a positive integer, which is a contradiction. Therefore, $a=1$. In this case, $m=nq-s$ and $n<s$, where $s\in \{5,7,35\}$ by \eqref{eq:case2-Sp4-s} and \eqref{eq:case2-n}. Therefore, $n$ is at most $4$, $6$ or $34$ respectively for $s=5$, $7$ or $35$. Moreover, for these values of $n$ and $s$, $q-1$ divides $7n^{2}-12ns+5s^{2}$. Therefore, $(s,q,n)$ is as in Table \ref{tbl:Sp4-snq} for which, by \eqref{eq:case2-Sp4-k}, we cannot find any possible parameters $k$ and $\lambda$.  Hence, $k$ is not a multiple of $q$.

\begin{table}[h]
    \centering
    \scriptsize
    \caption{Possible value of $(s,q,n)$ in Lemma \ref{tbl:Sp4-ap}.}\label{tbl:Sp4-snq}
\begin{center}
  \begin{tabular}{llll}
   \noalign{\smallskip}\hline\noalign{\smallskip}
    $s$ & $q$ & $n$ \\
    \hline\noalign{\smallskip}
    $5$ & $3$ & $1,3$ \\
    $5$ & $13$ & $1$ \\
    $5$ & $73$ & $1$ \\
    $7$ & $29$ &  $1,3$ \\
    $35$ & $43$ & $1,5,7,11,13,17,19,23,29,31$\\
    $35$ & $113$ & $1,3,9,11,17,19,27,33$\\
    $35$ & $463$ & $13$\\
    $35$ & $673$ & $19$\\
    $35$ & $883$ & $7$\\
    $35$ & $953$ & $1$\\
   \hline\noalign{\smallskip}
 \end{tabular}
\end{center}
\end{table}

Therefore, by Lemma~\ref{lem:New}(b) and~\eqref{eq:case2-Sp4-k}, we conclude that
\begin{align}\label{eq:Sp4-case2-1}
  m(v-1)+asf(q)\mid 2a^2sf(q)g_{1}(q),
\end{align}
where $g_{1}(q)=g(q)/q^{3}=2q(q^2-1)(q^4-1)$ and $f(q)= q-1$. Let now $d(q)=8q^3-2q^2-6q$ and $h(q)=2q^3-2q^2-2q$. Then
   $2a^2smf(q)g_{1}(q)-4a^2sh(q)[m(v-1)+asf(q)]=2a^2s[md(q)-2asf(q)h(q)]$.
It follows from \eqref{eq:Sp4-case2-1} that $v-1<2a^2s[|d(q)|+2as|f(q)h(q)|]$. This inequality holds only for pairs $(p,a)$ as in Table~\ref{tbl:Sp4-ap} below:
\begin{table}[h]
\centering\scriptsize
   \caption{Some parameters for Lemma~\ref{lem:Sp4}\label{tbl:Sp4-ap}}
  \begin{tabular}{ll}
   \noalign{\smallskip}\hline\noalign{\smallskip}
  $p$ & $a\leq $ \\
   \hline\noalign{\smallskip}
  $3$ & $12$\\
  $5$ & $6$\\
  $7$, $11$, $17$, $23$, $37$,$67$ & $3$\\
  $13$ & $4$ \\
  $29$, $41$, $43$, $71$ & $2$ \\
  $53$, $73, \ldots,19433$ & $1$\\
   \hline\noalign{\smallskip}
\end{tabular}
\end{table}

Again these values of $q=p^a$ do not give rise to any possible parameters, which is a contradiction.
\end{proof}

\begin{lemma}\label{lem:So4+}
The subgroup $H_{0}$ cannot be $^{\hat{}}{\SO_{4}^{+}}(q)\cdot d$ with $q\geq 5$ odd.
\end{lemma}
\begin{proof}
In this case, $|H_0|=q^2(q^2-1)^2$. Then by~\eqref{eq:v}, we have that $v=q^4(q^3+1)(q^2+1)/d$, where $d=\gcd(q+1, 4)$. Note in this case that $d$ is either $2$ or $4$.\smallskip

Suppose first $d=2$. It follows from \eqref{eq:k-out} that $k$ divides $4ag(q)$, where $g(q)=q^2(q^2-1)^2$. Moreover, Lemma~\ref{lem:six}(a) implies that $k$ divides $\lambda(v-1)$. Note that $v-1=[q^4(q^3+1)(q^2+1)-2]/2$. Since $\gcd(v-1, 4q^2(q^2-1)^2)=1$, we have that $k$ is a divisor of $\lambda a$. Then there exists a positive integer $m$ such that  $mk=\lambda a$. Since now $k(k-1)=\lambda(v-1)$, it follows that  $k= [m(v-1)/a]+1$, and since $k\mid 4ag(q)$, we must have $m(v-1)+a\mid 4a^{2}g(q)$. As $m\geq 1$, $v <4a^{2}g(q)=4a^2q^2(q^2-1)^2$, for $q$ odd, and this does not hold for any $q$, which is a contradiction.\smallskip

Suppose now $d=4$. Then $v=q^4(q^3-1)(q^2+1)/4$. Since $\gcd(v-1,8q^2(q^2-1)^2)$ divides $(q-1)^2$. Then $mk=\lambda a f(q)$, where $f(q)=(q-1)^2$ and $m$ is a positive integer. Thus $k= [m(v-1)/af(q)]+1$. As $k\mid 8ag(q)$, we must have $m(v-1)+af(q)\mid 8a^{2}g(q)f(q)$, and so $v <8a^{2}g(q)f(q)=8a^2q^2(q^2-1)^2(q-1)^2$, for $q$ odd. Thus $q\in \{3,7,11,19,23,27,243\}$, however, for these values of $q$, we cannot find any possible parameters.
\end{proof}

\begin{lemma}\label{lem:So4-}
The subgroup $H_{0}$ cannot be $^{\hat{}}{\SO_{4}^{-}}( q)\cdot d$ with $q$ odd.
\end{lemma}
\begin{proof}
In this case, $|H_0|=q^2(q^2+1)(q^2-1)$, and so by~\eqref{eq:v}, we have that $v=q^4(q^3+1)(q^2-1)/d$, where $d=\gcd(4,q+1)$. It follows from \eqref{eq:k-out} that $k$ divides $2ag(q)$, where $g(q)=q^2(q^4-1)$. Moreover, Lemma~\ref{lem:six}(a) implies that $k$ divides $\lambda(v-1)$.  As $q$ is odd, $d=2$ or $4$. Let $f(q)$ be $q-2$ if $d=2$, and $q-3$ if $d=4$. Then $\gcd(v-1,q^2(q^4-1))$ divides $f(q)$, and so $k$ is a divisor of $\lambda a f(q)$. Suppose that $m$ is a positive integer such that $mk=\lambda a f(q)$. Since now $k(k-1)=\lambda(v-1)$, it follows that  $k= [m(v-1)/a f(q)]+1$, and since $k\mid 4ag(q)$, we must have $m(v-1)+a f(q)\mid 2a^{2}dg(q)$. Therefore, $v <2a^{2}df(q)g(q)$, for $q$ odd, and this does not give rise to any possible parameters.
\end{proof}

\begin{lemma}\label{lem:S6}
The subgroup $H_{0}$ cannot be the subgroups as in the lines {\rm 11-16} of {\rm Table~\ref{tbl:maxes}}.
\end{lemma}
\begin{proof}
By Lemmas~\ref{lem:New}(b) and \ref{lem:six}(c), we have that $|X|\leq |\Out(X)|^{2}\cdot |H_{0}|^{3}$. Therefore, the lines {\rm 13-14} can be ruled out. For the remaining cases, this inequality holds only for $q$ listed as in Table~\ref{tbl:S6}. However, for such $q$, no divisor $k\geq 4$ of $|\Out(X)|\cdot |H_{0}|$ exists such that $k(k-1)/(v-1)$ is a positive integer, which is a contradiction.
\begin{table}[h]
  \centering\scriptsize
  \caption{Possible cases in Lemma~\ref{lem:S6}\label{tbl:S6}}
\begin{tabular}{ll}
   \noalign{\smallskip}\hline\noalign{\smallskip}
  \multicolumn{1}{c}{$H_{0}$} &\multicolumn{1}{l}{$q$}\\
   \hline\noalign{\smallskip}
  $^{\hat{}}(4\circ 2^{1+4})^{\cdot}S_{6}$ & $7$\\
  $^{\hat{}}(4\circ 2^{1+4})\cdot A_{6}$ &$3$ \\
  $^{\hat{}}{4_{2}^{\cdot}}\PSL_{3}(4)$ & $3$ \\
  $^{\hat{}}d\circ 2^{\cdot}\PSU_{4}(2)$ & $5,11$ \\
   \hline\noalign{\smallskip}
\end{tabular}
\end{table}
\end{proof}

\begin{proof}[\rm \textbf{Proof of Theorem~\ref{thm:main}}]
The proof of the main result follows immediately from Lemmas \ref{lem:su2}--\ref{lem:S6}.
\end{proof}

\section*{Acknowledgements}

The authors would like to thank anonymous referees for providing us helpful and constructive comments and suggestions. They would also like to mention that their names has been written in alphabetic order.

\section*{References}


\begin{thebibliography}{10}
    \expandafter\ifx\csname url\endcsname\relax
    \def\url#1{\texttt{#1}}\fi
    \expandafter\ifx\csname urlprefix\endcsname\relax\def\urlprefix{URL }\fi

    \bibitem{a:ABD-PSL3}
    S.~H. Alavi, M.~Bayat, Flag-transitive point-primitive symmetric designs and
    three dimensional projective special linear groups, Bulletin of Iranian
    Mathematical Society (BIMS) 42~(1) (2016) 201--221.

    \bibitem{a:ABD-PSL2}
    S.~H. Alavi, M.~Bayat, A.~Daneshkhah, Symmetric designs admitting
    flag-transitive and point-primitive automorphism groups associated to two
    dimensional projective special groups, Designs, Codes and Cryptography (2015)
    1--15.
    \newline\urlprefix\url{http://dx.doi.org/10.1007/s10623-015-0055-9}

    \bibitem{a:ABD-EXP}
    S.~H. Alavi, M.~Bayat, A.~Daneshkhah, Symmetric designs and finite simple
    exceptional groups of lie type, ArXiv e-prints (2017) 1--44.
    \newline\urlprefix\url{http://adsabs.harvard.edu/abs/2017arXiv170201257H}

    \bibitem{b:beth1999design}
    T.~Beth, D.~Jungnickel, H.~Lenz, Design Theory:, Design Theory, Cambridge
    University Press, 1999.
    \newline\urlprefix\url{https://books.google.com/books?id=VfVh8VaGGeYC}

    \bibitem{a:Braic-2500-nopower}
    S.~Brai{\'c}, A.~Golemac, J.~Mandi{\'c}, T.~Vu{\v{c}}i{\v{c}}i{\'c}, Primitive
    symmetric designs with up to 2500 points, J. Combin. Des. 19~(6) (2011)
    463--474.
    \newline\urlprefix\url{http://dx.doi.org/10.1002/jcd.20291}

    \bibitem{b:BHR-Max-Low}
    J.~N. Bray, D.~F. Holt, C.~M. Roney-Dougal, The maximal subgroups of the
    low-dimensional finite classical groups, vol. 407 of London Mathematical
    Society Lecture Note Series, Cambridge University Press, Cambridge, 2013,
    with a foreword by Martin Liebeck.
    \newline\urlprefix\url{http://dx.doi.org/10.1017/CBO9781139192576}

    \bibitem{a:Camina94}
    A.~R. Camina, A survey of the automorphism groups of block designs, J. Combin.
    Des. 2~(2) (1994) 79--100.
    \newline\urlprefix\url{http://dx.doi.org/10.1002/jcd.3180020205}

    \bibitem{b:Atlas}
    J.~H. Conway, R.~T. Curtis, S.~P. Norton, R.~A. Parker, R.~A. Wilson, Atlas of
    finite groups, Oxford University Press, Eynsham, 1985, maximal subgroups and
    ordinary characters for simple groups, With computational assistance from J.
    G. Thackray.

    \bibitem{a:D-PSU3}
    A.~Daneshkhah, S.~Zang~Zarin, Flag-transitive point-primitive symmetric designs and three dimensional projective special unitary groups, Bulletin of the Korean
    Mathematical Society. 54~(6) (2017) 2029–-2041.

    \bibitem{a:rank3}
    U.~Dempwolff, Primitive rank 3 groups on symmetric designs, Des. Codes
    Cryptogr. 22~(2) (2001) 191--207.
    \newline\urlprefix\url{http://dx.doi.org/10.1023/A:1008373207617}

    \bibitem{b:Dixon}
    J.~D. Dixon, B.~Mortimer, Permutation groups, vol. 163 of Graduate Texts in
    Mathematics, Springer-Verlag, New York, 1996.
    \newline\urlprefix\url{http://dx.doi.org/10.1007/978-1-4612-0731-3}

    \bibitem{GAP4}
    The GAP~Group, {GAP -- Groups, Algorithms, and Programming, Version 4.7.9}
    (2015).
    \newline\urlprefix\url{http://www.gap-system.org}

    \bibitem{a:Higman-rank3}
    D.~G. Higman, Finite permutation groups of rank {$3$}, Math. Z. 86 (1964)
    145--156.
    \newline\urlprefix\url{https://doi.org/10.1007/BF01111335}

    \bibitem{b:Hugh-design}
    D.~Hughes, F.~Piper, Design Theory, Up (Methuen), Cambridge University Press,
    1988.
    \newline\urlprefix\url{http://books.google.com/books?id=pAw5AAAAIAAJ}

    \bibitem{a:LSS1987}
    M.~W. Liebeck, J.~Saxl, G.~Seitz, On the overgroups of irreducible subgroups of
    the finite classical groups., Proc. Lond. Math. Soc. 50~(3) (1987) 507--537.
    \newline\urlprefix\url{http://dx.doi.org/10.1112/plms/s3-50.3.426}

    \bibitem{a:reg-classical}
    E.~O'Reilly-Regueiro, Biplanes with flag-transitive automorphism groups of
    almost simple type, with classical socle, J. Algebraic Combin. 26~(4) (2007)
    529--552.
    \newline\urlprefix\url{http://dx.doi.org/10.1007/s10801-007-0070-7}

    \bibitem{a:Praeger-45-12-3}
    C.~E. Praeger, The flag-transitive symmetric designs with 45 points, blocks of
    size 12, and 3 blocks on every point pair, Des. Codes Cryptogr. 44~(1-3)
    (2007) 115--132.
    \newline\urlprefix\url{https://doi.org/10.1007/s10623-007-9071-8}

    \bibitem{a:Saxl2002}
    J.~Saxl, On finite linear spaces with almost simple flag-transitive
    automorphism groups, J. Combin. Theory Ser. A 100~(2) (2002) 322--348.
    \newline\urlprefix\url{http://dx.doi.org/10.1006/jcta.2002.3305}

    \bibitem{a:tits}
    G.~M. Seitz, Flag-transitive subgroups of {C}hevalley groups, Ann. of Math. (2)
    97 (1973) 27--56.

    \bibitem{a:Zhou-lam100}
    D.~Tian, S.~Zhou, Flag-transitive point-primitive symmetric {$(v,k,\lambda)$}
    designs with {$\lambda$} at most 100, J. Combin. Des. 21~(4) (2013) 127--141.
    \newline\urlprefix\url{http://dx.doi.org/10.1002/jcd.21337}

    \bibitem{a:Zhou-lam-large-sporadic}
    D.~Tian, S.~Zhou, Flag-transitive $2$-$(v,k,\lambda)$ symmetric designs with
    sporadic socle, J. Combin. Des. 23 (2015) 140--150.
    \newline\urlprefix\url{http://dx.doi.org/10.1002/jcd.21385}

    \bibitem{a:Zhou-PSL2-q}
    D.~Tian, S.~Zhou, Classification of flag-transitive primitive symmetric
    {$(v,k,\lambda)$} designs with {${\rm PSL}(2,q)$} as socle, J. Math. Res.
    Appl. 36~(2) (2016) 127--139.

    \bibitem{a:Zhou-lam3-classical}
    S.~Zhou, H.~Dong, W.~Fang, Finite classical groups and flag-transitive
    triplanes, Discrete Math. 309~(16) (2009) 5183--5195.
    \newline\urlprefix\url{http://dx.doi.org/10.1016/j.disc.2009.04.005}

\end{thebibliography}

\end{document}